\documentclass[12pt,a4paper]{article}
\usepackage[latin1]{inputenc}
\usepackage{amsmath,amsthm,dsfont}
\usepackage{amsfonts}
\usepackage{amssymb}
\usepackage{graphicx}
\usepackage{hyperref,color,mathtools}
\usepackage{authblk}
\usepackage{geometry}
\newtheorem{theorem}{Theorem}
\newtheorem{lemma}{Lemma}

\newtheorem{remark}{Remark}
\mathtoolsset{showonlyrefs}

\DeclareMathOperator{\Tr}{Tr} 
\title{Limit theorems for loop soup random variables} 
\author[1,3]{Federico Camia \thanks{federico.camia@nyu.edu}}
\author[1,2]{Yves Le Jan \thanks{ yves.lejan@gmail.com}}
\author[1]{Tulasi Ram Reddy \thanks{tulasi@nyu.edu}}
\affil[1]{\it New York University Abu Dhabi, United Arab Emirates}
\affil[2]{\it New York University Shanghai, China}
\affil[3]{\it Vrije Universiteit Amsterdam, the Netherlands}
\date{}
\begin{document}

\maketitle

\begin{abstract}
This article deals with limit theorems for certain loop variables for loop soups whose intensity approaches infinity. We first consider random walk loop soups on finite graphs and obtain a central limit theorem when the loop variable is the sum over all loops of the integral of each loop against a given one-form on the graph. An extension of this result to the noncommutative case of loop holonomies is also discussed. As an application of the first result, we derive a central limit theorem for windings of loops around the faces of a planar graphs. More precisely, we show that the winding field generated by a random walk loop soup, when appropriately normalized, has a Gaussian limit as the loop soup intensity tends to $\infty$, and we give an explicit formula for the covariance kernel of the limiting field. We also derive a Spitzer-type law for windings of the Brownian loop soup, i.e., we show that the total winding around a point of all loops of diameter larger than $\delta$, when multiplied by $1/\log\delta$, converges in distribution to a Cauchy random variable as $\delta \to 0$.
\end{abstract}

\section{Introduction}
Windings of Brownian paths have been of interest since Spitzer's classic result \cite{spitzer} on their asymptotic behavior which states that, if $\theta(t)$ is the winding angle of a planar Brownian path about a point, then $2(\log t)^{-1}\theta(t)$ converges weakly to a Cauchy random variable as $t \rightarrow \infty$. The probability mass function for windings of any planar Brownian loop was computed in \cite{Yor_winding} (see also \cite{Yves_brownian}), and similar results for random walks were obtained in \cite{schapira2011}. Windings of simple random walks on the square lattice were more recently studied in \cite{budd_winding, budd_winding_2018}.

Symanzik, in his seminal work on Euclidean quantum field theories \cite{symanzik}, introduced a representation of a Euclidean field as a ``gas'' of (interacting) random paths. The noninteracting case gives rise to a Poissonian ensemble of Brownian loops, independently introduced by Lawler and Werner \cite{lawler_werner} who called it the \emph{Brownian loop soup}. Its discrete version, the \emph{random walk loop soup} was introduced in \cite{Lawler_ferraras}.

Integrals over one-forms for loops ensembles, which are generalizations of windings, were considered in \cite[Chapter-6]{Loops_stflour}. Various topological aspects of loop soups, such as homotopy and homology, were studied in \cite{LeJan_loops_topology}. In \cite{camia2016conformal}, the n-point functions of fields constructed taking the exponential of the winding numbers of loops from a Brownian loop soups are considered. The fields themselves are, a priori, only well-defined when a cutoff that removes small loops is applied, but the n-point functions are shown to converge to conformally covariant functions when the cutoff is sent to zero. A discrete version of these winding fields, based on the random walk loop soup, was considered in \cite{camia_lis_brug}. In that paper, the n-point functions of these discrete winding fields are shown to converge, in the scaling limit, to the continuum n-point functions studied in \cite{camia2016conformal}. The same paper contains a result showing that, for a certain range of parameters, the cutoff fields considered in \cite{camia2016conformal} converge to random generalized functions with finite second moments when the cutoff is sent to zero. A similar result was established later in \cite{Yves_brownian} using a different normalization and a different proof.

In this article, we focus mainly on loop ensembles on graphs (see \cite{Loops_stflour} for an introduction and various results on this topic), except for Section \ref{sec-4}, which deals with windings of the Brownian loop soup. In Section \ref{CLT-windings}, we establish a central limit theorem for random variables that are essentially sums of integrals of a one-form over loops of a random walk loop soup, as the intensity of the loop soup tends to infinity. In Section \ref{sec:application} we apply the results of Section \ref{CLT-windings} to the winding field generated by a random walk loop soup on a finite graph and on the infinite square lattice. Finally, in Section \ref{loop_holonomies}, we discuss an extension of this results of Section \ref{CLT-windings} to the noncommutative case of loop holonomies.

%

\section{A central limit theorem for loop variables}\label{CLT-windings}
Let $\mathcal{G}=(X,E)$ be a finite connected graph 
and, for any vertices $x,y \in X$, let $d(x,y)$ denote the graph distance between $x$ and $y$ and $d_x$ the degree of $x$. 
The transition matrix $P$ for the random walk on the graph $\mathcal{G}$ with killing function $\kappa: X \rightarrow [0,\infty)$ is given by
\begin{align}
P_{xy}=\begin{cases}
\frac{1}{\kappa_x+d_x} & \text{ if } d(x,y)=1, \\ 0 & \text{ otherwise.}
\end{cases} 
\end{align}
Let $G=(I-P)^{-1}$ denote the Green's function corresponding to $P$. $G$ is well defined as long as $\kappa$ is not identically zero.

We call a sequence  $\{x_0,x_1,\dots,x_n, x_{n+1}\}$ of vertices of $\mathcal{G}$ with $d(x_i,x_{i+1})=1$ for every $i=0,\dots,n$ and with $x_{n+1}=x_0$ a \emph{rooted loop} with root $x_0$ and denote it by $\gamma_r$. To each $\gamma_r$ we associate a weight $w_r(\gamma_r)=\frac{1}{n+1}P_{x_0x_1}\dots P_{x_nx_0}$. For a rooted loop $\gamma_r=\{x_i\}$, we interpret the index $i$ as time and define an \emph{unrooted loop} as an equivalence class of rooted loops in which two rooted loops belong to the same class if they are the same up to a time translation. To an unrooted loop $\gamma$ we associate a weight $\mu(\gamma)=\sum_{\gamma_r \in \gamma}w_r(\gamma_r) $. The \emph{random walk loop soup} $\mathcal{L}^\lambda$ with intensity $\lambda>0$ is a Poissonian collection of unrooted loops with intensity measure $\lambda\mu$.

A \emph{one-form} on $\mathcal{G}$ is a skew-symmetric matrix  $A$ with entries $A_{xy}=-A_{yx}$ if $d(x,y)=1$ and $A_{xy}=0$ otherwise. A special case of $A$ is illustrated in Figure \ref{Cuts_in_graph}. For any (rooted/unrooted) loop $\gamma=\{x_0,x_1,\dots,x_n,x_0\}$, denote $$\int\limits_{\gamma}A=A_{x_0,x_1}+A_{x_1,x_2}+\dots+A_{x_n,x_0}.$$  

Given a one-form $A$ and a parameter $\beta \in {\mathbb R}$, we define a `perturbed transition matrix' $P^{\beta}$ with entries
\begin{align}\label{perturbed_transition}
P_{xy}^{\beta}= \begin{cases}
\frac{e^{i\beta A_{xy}}}{\kappa_x+d_x} & \text{ if } d(x,y)=1, \\ 0 & \text{ otherwise.}
\end{cases}
\end{align}
Note that $P^{\beta}=P$ when $\beta=0$.

Our aim is to derive a central limit theorem for the loop soup random variable
\[
\int\limits_{\mathcal{L}^\lambda}A = \sum\limits_{\gamma\in \mathcal{L}^\lambda}{\int\limits_{\gamma}A}
\]
as the intensity $\lambda$ of the loop soup increases to infinity. The key to prove such a result is the following representation of the characteristic function of $\int\limits_{\mathcal{L}^\lambda}A$.
\begin{lemma}\label{partition_function}
With the above notation, assuming that $\kappa$ is not identically zero, we have that
$$\mathbb{E}_{\mathcal{L}^\lambda}\left[e^{\left(i\beta \int\limits_{\mathcal{L}^{\lambda}}A \right)}\right]=\left(\frac{\det(I-P^{\beta } )}{\det(I-P)}\right)^{-\lambda}.$$
\end{lemma}
\begin{proof}
Note that  $\det(I-P^{\beta})^\lambda$ is well defined and can be written as $e^{\lambda\log\det(I-P^{\beta})}=e^{\lambda\Tr\log(I-P^{\beta})}$. Since $\kappa $ is not identically $0$, the spectral radius of  $P^{\beta}$ is strictly less than $1$, which implies that
\[
-\log(I-P^{\beta})=\sum\limits_{k=1}^{\infty}\frac{(P^{\beta})^k}{k},
\]	
where the series in the above expression is convergent.

The weight of all loops of length $k\geq2$ is given by 
$\frac{1}{k}\Tr(P^k)$. Therefore the measure of all loops of arbitrary length is 
$$\sum\limits_{k=2}^\infty\frac{1}{k}\Tr(P^k)=-\Tr\log(I-P)=-\log\det(I-P).$$
Similarly, we have
\[
\int e^{i\beta\int\limits_{\gamma}A}d\mu(\gamma) = \sum\limits_{k=2}^\infty\frac{1}{k}\Tr((P^{\beta})^k)= -\log\det(I-P^{\beta}).
\]
Therefore, invoking Campbell's theorem for point processes, we have that
\begin{align*}
\mathbb{E}_{\mathcal{L}^\lambda}\left[e^{\left(i\beta\sum\limits_{\gamma\in \mathcal{L}^\lambda}{\int\limits_{\gamma}A}\right)}\right]  = e^{\left(\lambda\int [\exp{(i\beta\int\limits_{\gamma}A)}-1]d\mu(\gamma)\right)}
&= \frac{\det(I-P^{\beta})^{-\lambda}}{\det(I-P)^{-\lambda}},
\end{align*}
%
%
which concludes the proof.
\end{proof}

A way to interpret the lemma, which also provides an alternative proof, is to notice that $\det(I-P)^{-\lambda}$ is the partition function $Z_{\lambda}$ of the random walk loop soup on $\mathcal{G}$ with transition matrix $P$ and intensity $\lambda$, while $\det(I-P^{\beta})^{-\lambda}$ is the partition function $Z_{\lambda}^{\beta}$ of a modified random walk loop soup on $\mathcal{G}$ whose transition matrix is given by $P^{\beta}$. The expectation in Lemma \ref{partition_function} is given by $1/Z_{\lambda}=\det(I-P)^{\lambda}$ times the sum over all loop soup configurations $\mathcal{L}^\lambda$ of $\exp{\left(i\beta\sum\limits_{\gamma\in \mathcal{L}^\lambda}{\int\limits_{\gamma}A}\right)}$ times the weight of $\mathcal{L}^\lambda$. The factor $\exp{\left(i\beta\sum\limits_{\gamma\in \mathcal{L}^\lambda}{\int\limits_{\gamma}A}\right)}$ can be absorbed into the weight of $\mathcal{L}^\lambda$ to produce a modified weight corresponding to a loop soup with transition matrix $P^{\beta}$. Therefore the sum mentioned above gives the partition function $Z_{\lambda}^{\beta}=\det(I-P^{\beta})^{-\lambda}$. Other interpretations of the quantity in Lemma \ref{partition_function} will be discussed in the next section, after the proof of Lemma \ref{planar_n_point}.

To state our next result, we introduce the Hadamard and wedge matrix product operations denoted by $\odot$ and $\wedge$, respectively. For any two matrices U and V of same size, the Hardamard product between them (denoted $U \odot V$) is given by the matrix (of the same size as $U$ and $V$) whose entries are the products of the corresponding entries in $U$ and $V$. The following is the only property of matrix wedge products that will be used in this article: If $\lambda_1,\dots,\lambda_n$ are the eigenvalues of an $n \times n$ matrix $U$, then $\Tr(U^{\wedge k})=\sum\limits_{i_1<\dots<i_k}\lambda_{i_1}\dots\lambda_{i_k}$ for all $k \leq n$. Recall also that $G=(I-P)^{-1}$ denotes the Green's function corresponding to $P$, which is well defined as long as the killing function $\kappa$ is not identically zero.


\begin{theorem}\label{clt_one_form}
With the above notation, assuming that $\kappa$ is not identically zero, the distribution of the random variable $\frac{1}{\sqrt\lambda}\int\limits_{\mathcal{L}^\lambda}A=\frac{1}{\sqrt\lambda}\sum\limits_{\gamma\in \mathcal{L}^\lambda}{\int\limits_{\gamma}A}$ tends to a Gaussian distribution as $\lambda \to \infty$. More precisely,
\[
\lim\limits_{\lambda \rightarrow \infty}\mathbb{E}_{\mathcal{L}^\lambda}\left[e^{\left(i\frac{s}{\sqrt{\lambda}} \int\limits_{\mathcal{L}^{\lambda}} A \right)}\right] = \exp{\Big[-\frac{s^2}{2} \Big(\Tr\big((P\odot A^{\odot 2})G \big) - \Tr\big((P\odot A)G(P\odot A)G \big) \Big) \Big]}.
\]
\end{theorem}

\begin{proof}
Let $E^{\beta}=P-P^{\beta}$, then
\[
E_{xy}^{\beta}=\begin{cases}
\frac{1- e^{i\beta A_{xy}}}{\kappa_x+d_x} & \text{ if } d(x,y)=1, \\ 0 & \text{ otherwise.}
\end{cases}
\]
Invoking Lemma \ref{partition_function}, we have that 
\begin{eqnarray*}
\mathbb{E}_{\mathcal{L}^\lambda}\left[e^{\left(i\beta\sum\limits_{\gamma\in \mathcal{L}^\lambda}\int\limits_{\gamma}A\right)}\right] &=&
\left(\frac{\det(I-P^{\beta})}{\det(I-P)}\right)^{-\lambda} \nonumber \\
& = & \left(\frac{\det(I-P+E^{\beta})}{\det(I-P)}\right)^{-\lambda}, \nonumber \\
& = & \left(\det(I-P)^{-1}(I-P+E^{\beta})\right)^{-\lambda}, \nonumber \\
& = & \left(\det(I+(I-P)^{-1}E^{\beta})\right)^{-\lambda}, \nonumber \\
& = & \left(\det(I+GE^{\beta})\right)^{-\lambda}. \label{det_identity}
\end{eqnarray*}

Let $M$ be a square matrix  of dimension $n$ and $\|M\|$ denote the operator norm of $M$. Then,
\begin{align}
\det(I+M) &= 1+\Tr(M) +\Tr(M \wedge M) + \dots +\Tr(M^{\wedge n})
\end{align}
and
\begin{align}
|\Tr(M^{\wedge k})| & \leq  {n \choose k} \|M\|^k.
\end{align}
Using this, we can write
\begin{equation*}\label{mat_approx}
\left(\det(I+GE^{\beta})\right)^{-\lambda} = (1 + \Tr(GE^{\beta})+\Tr(GE^{\beta}\wedge GE^{\beta})+O(\beta^3\|A\|^3))^{-\lambda},
\end{equation*}
which leads to
\begin{align}
&\lim\limits_{\substack{\lambda \rightarrow \infty}} \mathbb{E}_{\mathcal{L}^\lambda}\Big[\exp{\Big(i\frac{s}{\sqrt{\lambda}}\sum\limits_{\gamma\in \mathcal{L}^\lambda}\int\limits_{\gamma}A\Big)}\Big]\\ & = \lim\limits_{\substack{\lambda \rightarrow \infty }}\left(1+\Tr(GE^{\frac{s}{\sqrt{\lambda}}})+\Tr(GE^{\frac{s}{\sqrt{\lambda}}}\wedge GE^{\frac{s}{\sqrt{\lambda}} })+O\left(\lambda^\frac{-3}{2}\right)\right)^{-\lambda}\\
&= \lim\limits_{\substack{\lambda \rightarrow \infty}} \exp{\left(-\lambda\log\left(1+\Tr(GE^{\frac{s}{\sqrt{\lambda}} })+\Tr(GE^{\frac{s}{\sqrt{\lambda}}}\wedge GE^{\frac{s}{\sqrt{\lambda}}})+O\left(\lambda^\frac{-3}{2}\right)\right)\right)}.
\end{align}

To preceed, note that  $\Tr(GE^{\beta}\wedge GE^{\beta })=\frac{1}{2}(\Tr^2(GE^{\beta })-\Tr(GE^{\beta } GE^{\beta }))$. Moreover, using the identities $G_{xy}P_{yx}=G_{yx}P_{xy}$ and $A_{xy}=-A_{yx}$ several times, one gets
\begin{align}
\Tr(GE^{\beta})&=\sum\limits_{x \thicksim y}G_{xy}E^{\beta}_{yx}\\
&= \frac{1}{2}\sum\limits_{x \thicksim y}(G_{xy}E^{\beta}_{yx}+G_{yx}E^{\beta}_{xy}),\\
&= \frac{1}{2}\sum\limits_{x \thicksim y}G_{xy}P_{yx}(1-e^{i\beta A_{yx}})+G_{yx}P_{xy}(1-e^{i\beta A_{xy}}),\\
&= \sum\limits_{x \thicksim y}G_{xy}P_{yx}(1-\cos(\beta A_{yx})).
\end{align}
Therefore,
\begin{align}\Tr(GE^{\frac{s}{\sqrt{\lambda}} A})&=\frac{s^2}{2\lambda}\sum\limits_{x \thicksim y}G_{xy}P_{yx}A_{yx}^2+O \left(\frac{1}{\lambda^\frac{3}{2}}\|A\|^3\right)\\&=\frac{s^2}{2\lambda}\Tr(G(P\odot A^{\odot 2}))+O\left(\frac{1}{\lambda^\frac{3}{2}}\|A\|^3\right).
\end{align}
Similarly,
%
%
\begin{align}
\Tr(GE^{\beta} GE^{\beta}) &= \sum\limits_{\substack{x_0,x_1,x_2,x_3 \\ x_0\thicksim x_1; x_2\thicksim x_3}}E^{\beta}_{x_0,x_1}G_{x_1x_2}E^{\beta}_{x_2,x_3}G_{x_3x_0},\\
 & = \sum\limits_{\substack{x_0,x_1,x_2,x_3\\ x_0\thicksim x_1; x_2\thicksim x_3}} (1-e^{i\beta A_{x_0x_1}})(1-e^{i\beta A_{x_2x_3}})P_{x_0x_1}P_{x_2x_3}G_{x_3x_0}G_{x_1x_2} ,\\
 & =\sum\limits_{\substack{x_0,x_1,x_2,x_3\\ x_0\thicksim x_1; x_2\thicksim x_3}} (i\beta A_{x_0x_1}+O(\beta^2\|A\|^2))(i\beta A_{x_2x_3}+O(\beta^2\|A\|^2))P_{x_0x_1}P_{x_2x_3}G_{x_3x_0}G_{x_1x_2},\\
 & =\sum\limits_{\substack{x_0,x_1,x_2,x_3\\ x_0\thicksim x_1; x_2\thicksim x_3}} -\beta^2 A_{x_0x_1}P_{x_0x_1}G_{x_1x_2}A_{x_2x_3}P_{x_2x_3}G_{x_3x_0} +O(\beta^3\|A\|^3),\\
 & =-\beta^2\Tr[(P\odot A)G(P\odot A)G] +O(\beta^3\|A\|^3). 
\end{align}
%
%
Note that the expressions $\Tr(GE^{\frac{1}{\sqrt{\lambda}}})$ and $\Tr(GE^{\frac{1}{\sqrt{\lambda}}}\wedge GE^{\frac{1}{\sqrt{\lambda}}})$ are of the order $\frac{1}{\lambda}$. Using this fact, and expanding the logarithm in power series, the above computations give
\begin{eqnarray*}
& \lim\limits_{\substack{\lambda \rightarrow \infty}} -\lambda\log\bigg(1+\Tr(GE^{\frac{s}{\sqrt{\lambda}} })+\Tr(GE^{\frac{s}{\sqrt{\lambda}} }\wedge GE^{\frac{1}{\sqrt{\lambda}} })+O\left(\lambda^\frac{-3}{2}\right)\bigg) \\
& \quad \quad \quad \quad \quad = \lim\limits_{\substack{\lambda \rightarrow \infty}} \left(-\frac{s^2}{2}\Tr(G(P\odot A^{\odot 2}))-\frac{s^2}{2}\Tr[(P\odot A)G(P\odot A)G] +O(\lambda^{\frac{-1}{2}})\right)\\
& = -\frac{s^2}{2}\Tr(G(P\odot A^{\odot 2}))-\frac{s^2}{2}\Tr[(P\odot A)G(P\odot A)G],
\end{eqnarray*}
which concludes the proof.
\end{proof}

\begin{remark} \label{remark}
{\rm
It may be useful to note the following identity, which holds when $A$ is skew-symmetric and $P$ is symmetric:
\begin{equation*}\label{trace_identity}
\frac12\sum\limits_{\substack{ x_0\thicksim x_1\\ x_2\thicksim x_3}}P_{x_0x_1}P_{x_2x_3}A_{x_0x_1}A_{x_2x_3}[G_{x_0x_3}G_{x_1x_2}-G_{x_0x_2}G_{x_1x_3}]=\Tr[(P\odot A)G(P\odot A)G].
\end{equation*}
We will use this identity in the proof of Theorem \ref{Gaussian_winding_field} in the next section. 
}
\end{remark}


\section{Central limit theorem for the loop soup winding field at high intensity} \label{sec:application}


The winding field generated by a loop soup on a planar graph $\mathcal{G} = (X,E)$ is defined on the faces $f$ of $\mathcal{G}$, which we identify with the vertices of the dual graph $\mathcal{G}^*=(X^*,E^*)$ (i.e., $f \in X^*$). Fix any face $f \in X^*$ and let $f_0=f,f_1,\dots,f_n$ be a sequence of distinct faces of $\mathcal{G}$ that are nearest-neighbors in $\mathcal{G}^*$, with $f_n$ the infinite face. The sequence $f_0,f_1,\dots,f_n$ determines a directed path $p$ from $f$ to the infinite face. Let $e^{p}_{i}$ denote the edge between $f_i$ and $f_{i+1}$ oriented in such a way that it crosses $p$ from right to left.
We let $\text{cut}(f)$ denote the collection of oriented edges $\{ e^{f}_{i} \}_{i=0}^{n-1}$. (See the Figure \ref{Cuts_in_graph} below for an example.)
Note that $\text{cut}(f)$ depends on the choice of $p$, but since all $p$'s connecting $f$ to the infinite face are equivalent for our purposes, we don't include $p$ in the notation.

Now take an oriented loop $\ell$ in $\mathcal{G}$ and assume that $\ell$ crosses $p$. In this case, we say that $\ell$ crosses $\text{cut}(f)$ and we call the crossing \emph{positive} if $\ell$ crosses $p$ from right to left and \emph{negative} otherwise.
For an oriented loop $\ell$ in $\mathcal{G}$ and a face $f \in X^*$, we define the \emph{winding number} of $\ell$ about $f$ to be
\begin{eqnarray*}
W_{\ell}(f)&= &  \text{number of positive  crossings of $\text{cut}(f)$ by $\ell$} \\& &-\text{ number of negative crossings of $\text{cut}(f)$ by $\ell$}
\end{eqnarray*}
for any choice of $\text{cut}(f)$. We note that $W_{\ell}(f)$ is well defined because the difference above is independent of the choice of $\text{cut}(f)$. (This is easy to verify and is left as an exercise for the interested reader.)

For a loop soup $\mathcal{L}^\lambda$, we define
\begin{equation} \label{def:winding field}
W_{\lambda} = \{ W_\lambda(f) \}_{f \in \mathcal{G}^*} = \Big\{ \sum_{\ell \in \mathcal{L}^\lambda}W_{\ell}(f) \Big\}_{f \in \mathcal{G}^*}
\end{equation}
to be the \emph{winding field} generated by $\mathcal{L}^\lambda$.

Theorem \ref{clt_one_form} can be used to prove a CLT for the winding field $W_{\lambda}$, when properly normalized, as $\lambda \to \infty$. In order to use Theorem \ref{clt_one_form}, we need a definition and a lemma. For any collection of faces $f_1,\dots,f_n$ of $\mathcal{G}$ and any vector $\bar{t}=(t_1,\ldots,t_n)$, define a skew-symmetric matrix $A^{\bar t}$ as follows. For each $i=1,\ldots,n$, choose a cut from $f_i$ to the infinite face as described above and denote it $\text{cut}(f_i)$. If $e=(x,y)$ is an edge of $\text{cut}(f_i)$ with positive orientation set $A^{\bar t}_{xy}=t_i$; if $e=(x,y)$ is an edge of $\text{cut}(f_i)$ with negative orientation set $A^{\bar t}_{xy}=-t_i$; otherwise set $A^{\bar t}_{xy}=0$. Note that one can write $A^{\bar t}$ as $A^{t_1}_{f_1}+\ldots+A^{t_n}_{f_n}$ where $A^{t_i}$ is a matrix such that $(A^{t_i}_{f_i})_{xy}=t_i$ if $(x,y)$ is in $\text{cut}(f_i)$ and has positive orientation, $(A^{t_i}_{f_i})_{xy}=-t_i$ if $(x,y)$ is in $\text{cut}(f_i)$ and has negative orientation, and $(A^{t_i}_{f_i})_{xy}=0$ if $(x,y) \notin \text{cut}(f_i)$.

\begin{lemma}\label{planar_n_point} For any collection of faces $f_1,\dots,f_n$ of $\mathcal{G}$, there exists a skew-Hermitian matrix $A^{\overline{t}}$ such that the characteristic function of the random vector $(W_{\lambda}(f_1),\ldots,W_{\lambda}(f_n))$ is given by 
 	$$\mathbb{E}_{\mathcal{L}^\lambda}[e^{i\beta(t_1 W_\lambda(f_1)+\dots+t_n W_\lambda(f_n))}]=\mathbb{E}_{\mathcal{L}^\lambda}\left[e^{\left(i\beta\sum\limits_{\gamma\in \mathcal{L}^\lambda}{\int\limits_{\gamma}A^{\overline{t}}}\right)}\right].$$
\end{lemma}

\begin{proof}
Using the matrices $A^{\bar t}$ describe above, the result follows immediately from the definition of winding number. 
\end{proof} 

The quantity in the lemma has several interpretations. Besides being the characteristic function of the random vector $(W_{\lambda}(f_1),\ldots,W_{\lambda}(f_n))$, it can be seen as the $n$-point function of a winding field of the type studied in \cite{camia_lis_brug} (see also \cite{camia2016conformal} for a continuum version). Moreover, by an application of Lemma \ref{partition_function},
\begin{equation*}
\mathbb{E}_{\mathcal{L}^\lambda}[e^{i(t_1 W_\lambda(f_1)+\dots+t_n W_\lambda(f_n))}] = \left(\frac{\det(I-P^{\bar{t}})}{\det(I-P)}\right)^{-\lambda},
\end{equation*}
where
\begin{align}
P_{xy}^{\bar{t}}= \begin{cases}
\frac{e^{i A^{\bar t}_{xy}}}{\kappa_x+d_x} & \text{ if } d(x,y)=1 \\ 0 & \text{ otherwise}
\end{cases}
\end{align}
and $A^{\bar t}$ is one of the matrices described above. A standard calculation using Gaussian integrals shows that
\begin{equation*}
Z^{\bar t}_{GFF} = \Pi_{x \in X} \Big(\frac{2\pi}{\kappa_x+d_x}\Big)^{1/2} \det(I-P^{\bar{t}})^{-1/2},
\end{equation*}
where $Z^{\bar t}_{GFF}$ is the partition function of the Gaussian Free Field (GFF) on $\mathcal{G}$ with Hamiltonian
\begin{equation} \label{Hamiltonian}
H^{\bar t}(\varphi) = -\frac12 \sum_{(x,y) \in E} e^{i A^{\bar t}_{xy}} \varphi_x \varphi_y + \frac12 \sum_{x \in X} (\kappa_x+d_x) \varphi_x^2.
\end{equation}
Hence, $\mathbb{E}_{\mathcal{L}^\lambda}[e^{i(t_1 W_\lambda(f_1)+\dots+t_n W_\lambda(f_n))}]$ can be written as a ratio of partition functions, namely,
\begin{equation*}
\mathbb{E}_{\mathcal{L}^\lambda}[e^{i(t_1 W_\lambda(f_1)+\dots+t_n W_\lambda(f_n))}] = \Big(\frac{Z^{\bar t}_{GFF}}{Z_{GFF}}\Big)^{2\lambda},
\end{equation*}
where $Z_{GFF}$ is the partition function of the `standard' GFF obtained from \eqref{Hamiltonian} by setting $t_1=\ldots=t_n=0$.

\begin{figure}[h!]
	\centering
	\includegraphics[width=.6\linewidth]{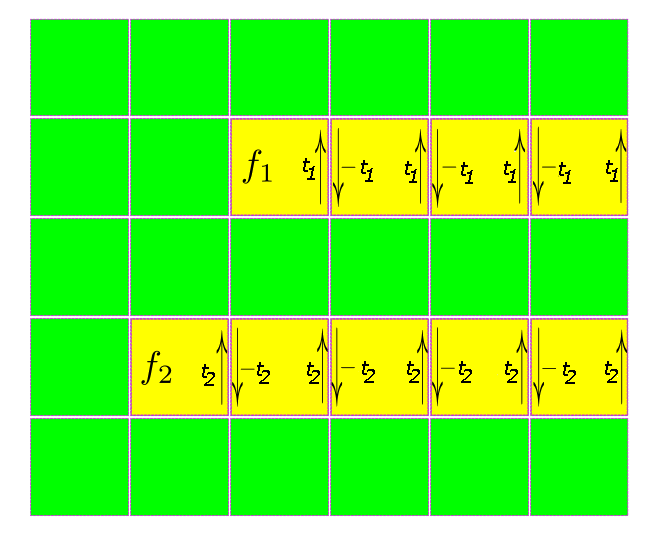}
	\label{Cuts_in_graph}
	\caption{The above figure displays a choice of cuts for faces $f_1$ and $f_2$ in a rectangular grid graph.}
\end{figure}

To state the next theorem we need some additional notation. For any directed edge $e \in \text{cut}(f)$, let $e^-$ and $e^+$ denote the starting and ending vertices of $e$, respectively.

\begin{theorem}\label{Gaussian_winding_field}
Consider a random walk loop soup on a finite graph $\mathcal{G}$ with symmetric transition matrix $P$ and the corresponding winding field $W_{\lambda}$. As $\lambda\rightarrow \infty$, $\frac{1}{\sqrt{\lambda}} W_{\lambda}$ converges to a Gaussian field whose covariance kernel is given by 
\begin{align}
K(f,g)=&\sum\limits_{e \in \text{cut}(f)}P_{e^+e^-}G_{e^+e^-}\mathds{1}_{f=g}\\&+2\sum\limits_{\substack{e_1\in \text{cut}(f)\\ e_2\in \text{cut}(g)}}P_{e_1^+e_1^-}P_{e_2^+e_2^-}\left(G_{e_1^+e_2^-}G_{e_1^-e_2^+}-G_{e_1^+e_2^+}G_{e_1^-e_2^-}\right)\mathds{1}_{f\neq g}.
\end{align}
\end{theorem}

\begin{proof}
Combining Lemma \ref{planar_n_point} and Theorem \ref{clt_one_form} shows that the winding field has a Gaussian limit as $\lambda \to \infty$:
\[
\Big\{ \frac{1}{\sqrt{\lambda}}W_\lambda(f) : f \text{ is a face of } \mathcal{G} \Big\} \xRightarrow[weakly]{\lambda\uparrow \infty} \Big\{ W(f) : f \text{ is a face of } \mathcal{G} \Big\}
\]
where $W(\cdot)$ is a Gaussian process on the faces of $G$.

Next, we compute the covariance kernel of the limiting Gaussian process. Choose two faces $f$ and $g$ and let $A^{\overline{t}}=A^{t_1}_{f}+A^{t_2}_{g}$, where $A^{t_1}_{f}$ has nonzero entries only along $\text{cut}(f)$ and $A^{t_2}_{g}$ has nonzero entries along $\text{cut}(g)$, as described above. Using Theorem \ref{clt_one_form} we obtain
\begin{align}
\lim\limits_{\lambda \rightarrow \infty}&\log\mathbb{E}_{\mathcal{L}^\lambda}[e^{i\frac{1}{\sqrt{\lambda}}(t_1 W_\lambda(f_1)+t_2 W_\lambda(f_2))}]\\ =&
-\frac{1}{2} \Big[ \Tr((P\odot (A^{\overline{t}})^{\odot 2})G) +\Tr((P\odot A^{\overline{t}})G(P\odot A^{\overline{t}})G) \Big]\\
=& -\frac{1}{2} \Big[ \Tr((P\odot (A^{t_1}+A^{t_2})^{\odot 2})G) +\Tr((P\odot (A^{t_1}+A^{t_2}))G(P\odot (A^{t_1}+A^{t_2})G) \Big]\\
=& -\frac{1}{2} \Big[ t_1^2 K(f,f) + t_2^2 K(g,g) - 2t_1t_2 K(f,g) \Big].
\end{align}
The variance of $W(f)$ is obtained by setting $t_1=t$ and $t_2=0$. In this case,
\begin{align}
K(f,f)&= \frac{1}{t^2} \Big[ \Tr((P\odot (A^{t})^{\odot 2})G) +\Tr((P\odot A^{t})G(P\odot A^{t})G) \Big] \\
&= \sum\limits_{e\in \text{cut}(f)}P_{e^-e^+}G_{e^+e^-} + \frac{1}{t^2}(\Tr((P\odot A^{t})G(P\odot A^{t})G)).
\end{align}
Since $P$ is assumed to be symmetric, the term $\Tr((P\odot A^{t})G(P\odot A^{t})G)$ vanishes. Therefore,
\[K(f,f)= \sum\limits_{e\in \text{cut}(f)}P_{e^-e^+}G_{e^-e^+}. \]

A similarly calculation, with $A^{\overline{t}}=A^{t_1}_{f}+A^{t_2}_{g}$, where $A^{t_1}_{f}$, gives the covariance:
\begin{eqnarray*}
K(f,g) & = & \frac{-1}{2t_1t_2}\Big[ \Tr((P\odot (A^{t_1}+A^{t_2})^{\odot 2})G) +\Tr(P\odot (A^{t_1}+A^{t_2})G(P\odot (A^{t_1}+A^{t_2})G)) \\
&& \quad - t_1^2 K(f,f) - t_2^2 K(g,g) \Big] \\
& = & \frac{-1}{2t_1t_2}\Big[\Tr((P\odot (A^{t_1}+A^{t_2})^{\odot 2})G) +\Tr(P\odot (A^{t_1}+A^{t_2})G(P\odot (A^{t_1}+A^{t_2})G)) \\
&& \quad - \Tr((P\odot (A^{t_1})^{\odot 2})G) -\Tr((P\odot A^{t_1})G(P\odot A^{t_1})G) \\
&& \quad - \Tr((P\odot (A^{t_2})^{\odot 2})G) -\Tr((P\odot A^{t_2})G(P\odot A^{t_2})G)\Big]\\
& = & \frac{-1}{t_1t_2}\Big[\Tr((P\odot (A^{t_1}\odot A^{t_2}))G)+\Tr((P\odot A^{t_1})G(P\odot A^{t_2})G)\Big].
\end{eqnarray*}
%
%
Since $P$ is symmetric, $\Tr((P\odot (A^{t_1}\odot A^{t_2}))G)=0$ and, using Remark \ref{remark}, we obtain
\[K(f,g)= 2 \sum\limits_{\substack{e_1\in \text{cut}(f)\\ e_2\in \text{cut}(g)}}P_{e_1^+e_1^-}P_{e_2^+e_2^-}\left(G_{e_1^+e_2^-}G_{e_1^-e_2^+}-G_{e_1^+e_2^+}G_{e_1^-e_2^-}\right),\]
which concludes the proof.
\end{proof}

\begin{remark} \label{alternative-proof}
{\rm We provide here an alternative, more direct, proof of Theorem \ref{Gaussian_winding_field}.
For any directed edge $e$, let $e^-$ and $e^+$ to denote the starting and ending vertices respectively. Moreover, let $N_e^+$ be the number of positive crossings of $e$ (i.e., from $e^-$ to $e^+$) by a loop from the loops soup and let $N_e^-$ be the number of negative crossings of $e$ (i.e., from $e^+$ to $e^-$) by a loop from the loops soup. The winding number about a face $f$ can be defined as, 
\[
W_\lambda(f)=\sum\limits_{e \in \text{cut}(f)}(N_e^+-N_e^-).
\]

Therefore the two point function for the winding numbers is given by
\begin{align}
\mathbb{E}_{\mathcal{L}^1}(W_1(f)W_1(g))&=\mathbb{E}_{\mathcal{L}^1}\Big(\sum\limits_{\substack{e_1\in \text{cut}(f)\\ e_2\in \text{cut}(g)}}(N_{e_1}^+-N_{e_1}^-)(N_{e_2}^+-N_{e_2}^-)\Big)\\
&= \mathbb{E}_{\mathcal{L}^1}\Big(\sum\limits_{\substack{e_1\in \text{cut}(f)\\ e_2\in \text{cut}(g)}}(N_{e_1}^+ N_{e_2}^+ + N_{e_1}^- N_{e_2}^- - N_{e_1} N_{e_2}^- - N_{e_1}^- N_{e_2}^+)\Big)\\
&= 2\sum\limits_{\substack{e_1\in \text{cut}(f)\\ e_2\in \text{cut}(g)}}\left(\mathbb{E}_{\mathcal{L}^1}(N_{e_1}^+N_{e_2}^+)-\mathbb{E}_{\mathcal{L}^1}(N_{e_1}^+N_{e_2}^-)\right)
\end{align}

Using this expression and a result from \cite{Loops_stflour} (see \cite[Exercise 10, Chapter 2]{Loops_stflour}, but note that in \cite{Loops_stflour} the Green's function is defined to be $[(I-P)^{-1}]_{x,y}/(\kappa_y+d_y) = P_{yx} G_{xy} = P_{xy} G_{xy}$, when $P$ is symmetric), we obtain
\begin{align}
\mathbb{E}_{\mathcal{L}^1}&(W_1(f)W_1(g))
\\
& = \sum\limits_{\substack{e_1\in \text{cut}(f)\\ e_2\in \text{cut}(g)}}\Big(P_{e_2^+e_2^-}P_{e_1^+e_1^-} G_{e_1^+e_2^-}G_{e_2^+e_1^-}+ P_{e_2^-e_2^+}P_{e_1^-e_1^+} G_{e_1^-e_2^+}G_{e_2^-e_1^+} \\
& \quad \quad \quad - P_{e_2^-e_2^+}P_{e_1^+e_1^-} G_{e_1^+e_2^+}G_{e_2^-e_1^-} - P_{e_2^+e_2^-}P_{e_1^-e_1^+} G_{e_1^-e_2^-}G_{e_2^+e_1^+}\Big)
\\
&=2\sum\limits_{\substack{e_1\in \text{cut}(f)\\ e_2\in \text{cut}(g)}} P_{e_1^+e_1^-}P_{e_2^+e_2^-} \left(G_{e_1^+e_2^-}G_{e_2^+e_1^-}-G_{e_1^+e_2^+}G_{e_2^-e_1^-}\right).
\end{align}

Now take $\lambda=n \in {\mathbb N}$ and note that, for any face $f$, $W_n(f)$ is distributed like $\sum_{i=1}^n W_1^i(f)$, where $\{ W_1^i(f) \}_{i=1,\ldots,n}$ are $n$ i.i.d.\ copies of $W_1(f)$. Therefore, for any collection of faces $f_1,\dots,f_m$, the central limit theorem implies that, as $\lambda=n \to \infty$, the random vector $\frac{1}{\sqrt{n}}\left(W_n(f_1),\dots,W_n(f_m)\right)$ converges to a multivariate Gaussian with covariance kernel given by  the two-point function $\mathbb{E}_{\mathcal{L}^1}(W_1(f)W_1(g))$ calculated above.
}
\end{remark}

\begin{remark} \label{square-lattice}
{\rm
Theorem \ref{Gaussian_winding_field} can be extended to infinite graphs, as we now explain. For concreteness and simplicity, we focus on the square lattice and consider a random walk loop soup with constant killing function: $\kappa_x=\kappa>0$ for all $x \in {\mathbb Z}^2$. Note that in this case the transition matrix $P$ and the Green's function $G$ are symmetric. Moreover, contrary to the case $\kappa=0$, the winding field of the random walk loop soup on ${\mathbb Z}^2$ is well defined when $\kappa>0$.
To see this, note that, since the loop soup is a Poisson process, we can bound the expected number of loops intersecting $(-a,0),(b,0) \in {\mathbb Z}^2$ as follows:
\begin{eqnarray*}
\mathbb{E}_{\mathcal{L}^{\lambda}}\big(\# \text{ loops joining } (-a,0) \text{ and }(b,0) \big) & = & \lambda \mu\big(\gamma: \gamma (-a,0),(b,0) \in \gamma \big) \\
& \leq & \lambda \sum_{m \geq 2(a+b)} \Big( \frac{4}{\kappa+4} \Big)^m \\
& = & \frac{\lambda \kappa}{4} \Big( 1 + \frac{\kappa}{4} \Big)^{-2(a+b)}.
\end{eqnarray*}
Hence, the expected number of loops winding around the origin is bounded above by $\frac{\lambda\kappa}{4} \sum_{a=1}^{\infty}\sum_{b=1}^{\infty} (1+\kappa/4)^{-2(a+b)} < \infty$ for any $\kappa>0$. This means that, with probability one, the number of loops winding around any vertex is finite. Because of this, one can obtain the winding field on ${\mathbb Z}^2$ as the weak limit of winding fields in large finite graphs $\mathcal{G}_n=[-n,n]^2\cap \mathbb{Z}^2$ as $n \to \infty$. It is now clear that one can apply the arguments in Remark \ref{alternative-proof} to the case of the winding field on $\mathbb{Z}^2$.
}
\end{remark}

\section{A Spitzer-type law for windings of the Brownian loop soup}\label{sec-4}
Spitzer showed \cite{spitzer} that the winding of Brownian motion about a given point up to time $t$, when scaled by $1/(2\log t)$, converges in distribution to a Cauchy random variable as $t \rightarrow \infty$. An analogous result for the simple symmetric random walk on ${\mathbb Z}^2$ is contained in Theorem 6 of \cite{budd_winding_2018}. In this section we prove a similar result for the Brownian loop soup in a bounded domain.

Recall that the Brownian loop soup in a planar domain $D \subset {\mathbb C}$ is defined as a Poisson process of loops with intensity measure $\mu^{\text{loop}}$ given by 
\[
\mu^{\text{loop}}(\cdot)=\int\limits_{D}\int\limits_{0}^{\infty}\frac{1}{2\pi t^2}\mu_{\text{BB}}^{z,t}(\cdot)dtdA(z),
\]
where $\mu_{\text{BB}}^{z,t}$ is the Brownian Bridge measure of time length $t$ starting at $z$ and $dA$ is the area measure on the complex plane (see \cite{lawler_werner} for a precise definition). 

For any $z \in D$, we let $W^\delta_\lambda(z)$ denote the sum of the winding numbers about $z$ of all Brownian loops contained in $D$ with diameter at least $\delta$ for some $\delta>0$.
\begin{theorem}
Consider a bounded domain $D \subset {\mathbb C}$. For any $z \in D$, as $\delta \rightarrow 0$, $\frac{W^\delta_\lambda(z)}{\log\delta}$ converges weakly to a Cauchy random variable with location parameter $0$ and scale parameter $\frac{\lambda}{2\pi}$.
\end{theorem}
\begin{proof}
Let $d_z$ denote the distance between $z$ and the boundary of $D$ and, for $\delta<d_z$, let $W^{\delta,d_z}_\lambda(z)$ denote the sum of the winding numbers about $z$ of all Brownian loops with diameter between $\delta$ and $d_z$. Note that, because of the Poissonian nature of the Brownian loop soup, the random variables $W^{\delta,d_z}_\lambda(z)$ and $W^{d_z}_\lambda(z)$ are independent.

The key ingredient in the proof is Lemma 3.2 of \cite{camia2016conformal}, which states, in our notation, that
\[
\mathbb{E}\big(e^{i \beta W^{\delta,d_z}_\lambda(z)}\big) = \Big( \frac{d_z}{\delta} \Big)^{-\lambda\frac{\beta(2\pi-\beta)}{4\pi^2}}
= d_z^{-\lambda\frac{\beta(2\pi-\beta)}{4\pi^2}} e^{\lambda\frac{\beta(2\pi-\beta)}{4\pi^2}\log\delta}
\]
when $\beta \in [0,2\pi)$, and that the same expression holds with $\beta$ replaced by $(\beta \mod 2\pi)$ when $\beta \notin [0,2\pi)$.
With this result, choosing $\beta=s/\log\delta$, the limit as $\delta \to 0$ of the characteristic function of $\frac{W^{\delta}_\lambda(z)}{\log\delta}$ can be computed as follows:
\[
\lim_{\delta \to 0} \mathbb{E}\big(e^{i \frac{s}{\log\delta} W^\delta_\lambda(z)}\big) = \lim_{\delta \to 0} \mathbb{E}\big(e^{i \frac{s}{\log\delta} W^{\delta,d_z}_\lambda(z)}\big) \mathbb{E}\big(e^{i \frac{s}{\log\delta} W^{d_z}_\lambda(z)}\big) = e^{-\frac{\lambda}{2\pi}|s|},
\]
where the right hand side is the characteristic function of a Cauchy random variable with location parameter $0$ and scale parameter $\frac{\lambda}{2\pi}$.
\end{proof}

\section{Holonomies of loop ensembles}\label{loop_holonomies}
Theorem \ref{clt_one_form} can be generalized to loop holonomies. Assume that the transition matrix $P$ introduced at the beginning of Section \ref{CLT-windings} is symmetric and hence the Green's function is also symmetric. We consider a connection on the graph $\mathcal{G}$, given by assigning to each oriented edge $(x,y)$ a $d\times d$ unitary matrix $\bf {U}_{xy}$ of the form $U_{xy}=e^{iA_{xy}}$ for some Hermitian matrix $A_{xy}$. For any closed loop $\gamma=\{x_0,x_1,\dots,x_n,x_0\}$, we denote $$\prod\limits_{\gamma}{\bf U=U_{x_0x_1}U_{x_1x_2}\dots U_{x_nx_0}}.$$
We also write $\Tr_\gamma\left[{\bf U}\right]$ for $\Tr[{\bf U_{x_0x_1}}{\bf U_{x_1x_2}}\dots {\bf U_{x_{n}x_0}}]$, which is well defined as the expression inside $\Tr[\cdot]$ is shift invariant. We will re-do the computations leading to Theorem \ref{clt_one_form}, in this case by invoking block matrices. Note that since ${\bf U}_{xy} = {\bf U}_{yx}^{-1}$, we assume ${\bf A}_{xy}=-{\bf A}_{yx}$.  Denote the corresponding block matrix whose blocks are ${\bf A_{xy}}$ with ${\bf A}$. Similarly denote $\Tr_\gamma\left[e^{i\beta\bf{A}}\right]:=\Tr\left[e^{i\beta {\bf A_{x_0x_1}}}\dots e^{i\beta {\bf A_{x_nx_0}}}\right]$. We denote the tensor product between two matrices $A$ and $B$ to be $A \otimes B$ and the Hadamard product to be $A \odot B$.

In this context, the quantity
$\exp{\left(i\sum\limits_{\gamma\in \mathcal{L}^\lambda}\frac{1}{\sqrt{\lambda}}\int\limits_{\gamma}A\right)} = \prod_{\gamma\in \mathcal{L}^\lambda} e^{\frac{i}{\sqrt{\lambda}}\int\limits_{\gamma}A}$
that appears in Theorem \ref{clt_one_form} will be replaced by $\prod\limits_{\gamma \in \mathcal{L}^\lambda}\frac{1}{d}\Tr_\gamma\left(e^{i\frac{1}{\sqrt{\lambda}}{\bf A}}\right)$. The expectation of this quantity cannot be interpreted as a characteristic function, but other interpretations such as those discusses after Lemma \ref{partition_function} and Lemma \ref{planar_n_point} are still available.

The first step towards the main result of this section is the following lemma.

\begin{lemma}\label{partition-function-non-commutative}
With the above notation we have
$$\mathbb{E}_{\mathcal{L}^\lambda}\left[\prod\limits_{\gamma \in \mathcal{L}^\lambda}\frac{1}{d}\Tr_\gamma\left(e^{i\beta\bf{A}}\right)\right]=\left(\frac{\det(I_{nd}-(P\otimes J_d )\odot \bf{U}_\beta)}{\det(I_{nd}-P\otimes I_d)}\right)^{-\lambda},$$ 
where $(P\otimes  J_d) \odot {\bf U}$ and $P \otimes I_d$ are block matrices whose blocks are $P_{ij}{\bf U_{ij}}$ and $P_{ij}I_d$ respectively, $J_d$ is $d \times d$ matrix whose entries are all $1$ and for any $k$, $I_k$ is the $k \times k$ identity matrix.
\end{lemma}
\begin{proof}	
The statement follows from a computation similar to that in the proof of Lemma \ref{partition_function}, namely 
\begin{align*}
\mathbb{E}_{\mathcal{L}^\lambda}\left[\prod\limits_{\gamma \in \mathcal{L}^\lambda}\frac{1}{d}\Tr_\gamma\left(e^{i\beta\bf{A}}\right)\right]
&= \frac{\exp(-\lambda\sum\limits_{k=1}^{\infty}\frac{1}{k}\Tr((P\otimes  J_d) \odot {\bf U}_\beta)^k)}{\exp(-\lambda\sum\limits_{k=1}^{\infty}\frac{1}{k}\Tr(P\otimes I_d)^k)}\\ 
&= \frac{e^{\lambda\Tr\log(I-(P\otimes  J_d) \odot {\bf U}_\beta)}}{e^{\lambda\Tr\log(I-P\otimes I_d)}}\\
&= \frac{(\det(I-(P\otimes  J_d) \odot{\bf U}_\beta))^{-\lambda}}{(\det(I- P\otimes I_d))^{-\lambda}}.
\end{align*}
For readers interested in more details, we note that a similar computation can be found in \cite[Proposition 23]{Loops_stflour}.
\end{proof}

We are now ready to state and prove the main result of this section.
\begin{theorem}\label{clt_representation}
With the notation above we have
\begin{align}
&	\lim\limits_{\lambda \rightarrow \infty} \mathbb{E}_{\mathcal{L}^\lambda}\left[\prod\limits_{\gamma \in \mathcal{L}^\lambda}\frac{1}{d}\Tr_\gamma\left(e^{i\frac{1}{\sqrt{\lambda}}{\bf A}}\right)\right]=\\&\exp{\Big(
	-\frac{1}{2}\Big[\sum\limits_{x \thicksim y}G_{xy}P_{xy}\Tr({\bf A}_{xy}^2)+ \sum\limits_{\substack{ x_0\thicksim x_1\\ x_2\thicksim x_3}}P_{x_0x_1}P_{x_2x_3}\Tr[{\bf A}_{x_0x_1}{\bf A}_{x_2x_3}] (G_{x_0x_3}G_{x_1x_2}-G_{x_0x_2}G_{x_1x_3})\Big]\Big)}.
\end{align}
\end{theorem}
\begin{proof}
We follow the computation in the proof of Theorem \ref{clt_one_form}. From Lemma \ref{partition-function-non-commutative}, defining ${\bf E}^{\beta{\bf A}} = P \otimes { I}_d - (P \otimes  J_d )\odot {\bf U}_\beta$, we have
\begin{align*}
	 \mathbb{E}_{\mathcal{L}^\lambda}\left[\prod\limits_{\gamma \in \mathcal{L}^\lambda}\frac{1}{d}\Tr_\gamma\left(e^{i\beta{\bf A}}\right)\right] &= \left(\frac{\det(I_{nd}-(P\otimes  J_d) \odot \bf{U}_\beta)}{\det(I_{nd}-P\otimes I_d)}\right)^{-\lambda}\\
	&= (\det(I_{nd}+(I_{nd}-P\otimes{ I}_d)^{-1}{\bf E}^{\beta A }))^{-\lambda}.
\end{align*}
Therefore,
\begin{align}\label{det_ratio_tensor}
&\lim\limits_{\lambda \rightarrow \infty}\left(	\frac{\det(I_{nd}-(P\otimes  J_d) \odot {\bf U}_\frac{1}{\sqrt{\lambda}})}{\det(I_{nd}-P\otimes I_d)}\right)^\lambda \\ &= \lim\limits_{\lambda \rightarrow \infty} \bigg(1 + \Tr((I_{nd}-P\otimes I_d)^{-1}{\bf E}^{\frac{1}{\sqrt{\lambda}}  {\bf A}})\\& \hspace{2 cm} +\Tr((I_{nd}-P\otimes I_d)^{-1}{\bf E}^{\frac{1}{\sqrt{\lambda}}  {\bf A}}\wedge(I_{nd}-P\otimes I_d)^{-1}{\bf E}^{\frac{1}{\sqrt{\lambda}}  {\bf A}}) +O\bigg(\frac{1}{\lambda^\frac{3}{2}}\| {\bf A}\|^3\bigg)\bigg)^\lambda\\
&= \lim\limits_{\lambda \rightarrow \infty} \exp{\Big[ \lambda \Big( \Tr\big((I_{nd}-P\otimes I_d)^{-1}{\bf E}^{\frac{1}{\sqrt{\lambda}} {\bf A}}\big) + \Tr\big((I_{nd}-P\otimes I_d)^{-1}{\bf E}^{\frac{1}{\sqrt{\lambda}}  {\bf A}}\wedge(I_{nd}-P\otimes I_d)^{-1}{\bf E}^{\frac{1}{\sqrt{\lambda}} {\bf A}}\big) \Big) \Big]}.
\end{align}

Note that $(I_{nd}-P\otimes I_d)^{-1}= ((I_n-P)\otimes I_d)^{-1}=G\otimes I_d$. Moreover, expanding the traces of block matrices in terms of traces of blocks, we have
\begin{align}
\Tr((G\otimes I_d){\bf E}^{\beta {\bf A}}) &= \sum\limits_{x \thicksim y}G_{xy}P_{yx}\Tr(I_d-e^{i\beta{\bf A}_{yx}})\\
& = -\frac{\beta^2}{2}\sum\limits_{x \thicksim y}G_{xy}P_{xy}\Tr({\bf A}_{xy}^2) +O(\beta^3\|{\bf A}\|^3_\infty). \label{one_point_non-commutative}
\end{align}
Similarly,

\begin{align}
\Tr({\bf E}^{\beta  {\bf A}}(G\otimes I_d){\bf E}^{\beta  {\bf A}}(G\otimes I_d)) ={}& \sum\limits_{\substack{ x_0\thicksim x_1\\ x_2\thicksim x_3}}\Tr(E^{\beta {\bf A}_{x_0,x_1}}G_{x_1x_2}E^{\beta {\bf A}_{x_2,x_3}}G_{x_3x_0}),\\
=& \sum\limits_{\substack{ x_0\thicksim x_1\\ x_2\thicksim x_3}}\Tr[E^{\beta {\bf A}_{x_0x_1}}G_{x_1x_2}E^{\beta {\bf A}_{x_2x_3}}G_{x_3x_0}]\\
=& \sum\limits_{\substack{ x_0\thicksim x_1\\x_2\thicksim x_3}}\Tr[(I_d-e^{i\beta{\bf A}_{x_0x_1}})(I_d-e^{i\beta{\bf A}_{x_2x_3}})P_{x_0x_1}P_{x_2x_3}G_{x_0x_3}G_{x_1x_2}]\\
=& \sum\limits_{\substack{ x_0\thicksim x_1\\ x_2\thicksim x_3}}P_{x_0x_1}P_{x_2x_3}\Tr[-\beta^2 {\bf A}_{x_0x_1}{\bf A}_{x_2x_3}G_{x_0x_3}G_{x_2x_1}]
 +O(\beta^3\|{\bf A}\|^3_\infty)\\
=& -\frac{\beta^2}{2} \sum\limits_{\substack{ x_0\thicksim x_1\\ x_2\thicksim x_3}}P_{x_0x_1}P_{x_2x_3}\Tr[{\bf A}_{x_0x_1}{\bf A}_{x_2x_3}] (G_{x_0x_3}G_{x_1x_2}-G_{x_0x_2}G_{x_1x_3}) \\&+O(\beta^3\|{\bf A}\|^3_\infty).
\end{align}

Invoking the identity $\Tr\big(M\wedge M)= \frac{1}{2}(\Tr(M)^2-\Tr(M^2)\big)$ and the computation above, we have that
\begin{align}
\lim\limits_{\lambda \rightarrow \infty}&\log \mathbb{E}_{\mathcal{L}^\lambda}\left[\prod\limits_{\gamma \in \mathcal{L}^\lambda}\frac{1}{d}\Tr_\gamma\left(e^{i\frac{1}{\sqrt{\lambda}}{\bf A}}\right)\right]\\ =&
-\frac{1}{2}\sum\limits_{x \thicksim y}G_{xy}P_{xy}\Tr({\bf A}_{xy}^2)\\& -\frac{1}{2} \sum\limits_{\substack{ x_0\thicksim x_1\\ x_2\thicksim x_3}}P_{x_0x_1}P_{x_2x_3}\Tr[{\bf A}_{x_0x_1}{\bf A}_{x_2x_3}] (G_{x_0x_3}G_{x_1x_2}-G_{x_0x_2}G_{x_1x_3}),
\end{align}
which concludes the proof.
\end{proof}

\bigskip

\bigbreak\noindent\textbf{Acknowledgments.}
FC thanks David Brydges for an enlightening discussion during the workshop ``Random Structures in High Dimensions'' held in June-July 2016 at the Casa Matem\'atica Oaxaca (CMO) in Oaxaca, Mexico.

\bibliographystyle{alpha}
\bibliography{loop_soups}
\end{document}